\documentclass[preprint,12pt]{elsarticle}

\usepackage{amssymb, amsmath,amsthm}
\usepackage{graphicx}
\usepackage{tikz}
\usepackage{enumerate}
\newtheorem{theorem}{Theorem}

\newdefinition{example}{Example}
\newtheorem{proposition}[theorem]{Proposition}
\newtheorem{corollary}[theorem]{Corollary}

\tikzstyle{every node}=[fill=white,
                        inner sep=0pt, minimum width=10pt]

\journal{.}

\begin{document}

\begin{frontmatter}


 \author{Curtis Nelson\corref{cor1}}
  \ead{curtisgn@gmail.com}
  
  \author{Bryan Shader\corref{cor2}}
  \ead{bshader@uwyo.edu}
\tnotetext[fun]{This research was supported by the University of Wyoming.}
\cortext[cor1]{Corresponding author.}

\title{All pairs suffice \tnoteref{fun}}


\address{Department of Mathematics, University of Wyoming, Laramie, Wyoming, USA}

\begin{abstract}
A P-set of a symmetric matrix $A$ is a set $\alpha$ of indices such that the nullity of the matrix obtained from $A$ by removing rows and columns indexed by $\alpha$ is $|\alpha|$ more than that of $A$.  It is known that each subset of a P-set is a P-set.  It is also known that a set of indices such that each singleton subset is a P-set need not be a P-set.  This note shows that if all pairs of vertices of a set with at least two elements are P-sets, then the set is a P-set.

\end{abstract}

\begin{keyword}
acyclic matrix \sep eigenvalues \sep P-vertex \sep P-set \sep Parter-vertex \sep Parter-Set 

\MSC 15A18 \sep 05C50

\end{keyword}

\end{frontmatter}



\section{Introduction}


Throughout this note, all matrices are real.  Let $A = [a_{ij}]$ be an $\text{n} \times \text{n}$ symmetric matrix.  The graph, $G(A)$, of $A$ has vertices $1,2,\dots, n$ where $\{i,j\}$ is an edge in the graph if and only if $a_{ij} \neq 0$.  Let $\tau$ and $\alpha$ be subsets of $\{1,2,\dots, n\}$.  Then $A[\tau, \alpha]$ denotes the submatrix of $A$ consisting of the rows indexed by $\tau$ and the columns indexed by $\alpha$.  When $\tau = \alpha$, $A[\alpha]$ is used in place of $A[\alpha, \alpha]$.  Similarly, $A(\alpha)$ denotes the principal submatrix of $A$ obtained by deleting the rows and columns indexed by $\alpha$.  Also, $\nu(A)$ denotes the nullity of $A$ and RS$(A)$ denotes the row space of $A$.  

Vertex $i$ is a \emph{downer vertex of $A$} if $\nu(A)-1 = \nu(A(i))$.  Vertex $i$ is a \emph{P-vertex of $A$} if $\nu(A)+1 = \nu(A(i))$.  A set of vertices indexed by a set $\alpha$ is a \emph{P-set of $A$} if $\nu(A)+|\alpha| = \nu(A(\alpha))$.  In this case, we also say that $\alpha$ is a P-set of $A$.  It is known that every non-empty subset of a P-set is a P-set \cite[Proposition 5]{KS1}.  It is also known that a set of P-vertices is not necessarily a P-set \cite[Example 2.4]{JDSSW}, \cite[Example 4.6]{KS}.  In \cite{NS2}, we studied the maximal P-sets of matrices whose graphs are trees and also proved that for a symmetric matrix $A$ whose graph is a tree, a set of indices is a P-set of $A$ if and only if every subset of cardinality two (i.e.~each pair) is a P-set of $A$.  This conclusion came as a corollary of a complicated, technical theorem that identifies the maximal P-sets of $A$.  In this note, we show that the previous result holds for all symmetric matrices $A$, regardless of the graph of $A$; i.e.~we show that that if $A$ is a symmetric matrix, a set $\alpha$ of at least two indices is a P-set of $A$ if and only if each pair of $\alpha$ is a P-set of $A$.

\section{All pairs suffice}\label{ourresults1}

We first recall the following known result, Jacobi's Determinant Identity \cite[p. 24]{HJ}.

\begin{theorem}\label{Jacobi}
Let $A$ be an invertible $n \times n$ matrix and let $\alpha$ be a subset of $\{1,2,\dots, n\}$.  Then $\det(A[\alpha]) = \det(A^{-1}(\alpha))\det(A)$.
\end{theorem}

We next prove a result that shows useful relationships between the linear dependencies among the rows of a matrix and the downer vertices and P-vertices of the matrix. 

\begin{proposition}\label{independent}
Let $A$ be a symmetric matrix.  Then 

\item{\rm (a)} $i$ is a downer vertex if and only if row $i$ is a linear combination of other rows of $A$.
\item{\rm (b)} If $\alpha$ is a set of P-vertices, then the rows of $A$ indexed by $\alpha$ are linearly independent.

\end{proposition}

\begin{proof}
Vertex $i$ is a downer vertex of $A$ if and only if rank$(A)$ = rank$(A(i))$.  This occurs if and only if row $i$ is a linear combination of other rows of $A$.

For (b), we prove the contrapositive.  Assume that the rows of $A$ indexed by $\alpha$ are linearly dependent.  Then there is an $i \in \alpha$ such that row $i$ is a linear combination of rows of $A$ indexed by $\alpha\setminus\{i\}$.  By $(a)$, $i$ is a downer vertex of $A$ and hence is not a P-vertex.
\end{proof}

The following gives useful characterizations of a P-set.

\begin{theorem}\label{Theorem1}
Let $$A = 
\left[ \begin{array}{c | c}
B & C \\ \hline
C^T & D
\end{array} \right]$$ be a symmetric $n \times n$ matrix where $B$ is $k \times k$.  Then the following are equivalent:

\item{\rm (a)} $\{1,2,\dots, k\}$ is a P-set
\item{\rm (b)} rank$(A)$ = rank$(D)$ $+$ $2k$
\item{\rm (c)} $\{x: x^TC \in \rm{RS}(D)\} = \{0\}$

\end{theorem}

\begin{proof}
Note that $\{1,2,\dots,k\}$ is a P-set of $A$ 
\begin{align*}
&\Leftrightarrow           &  \nu(D) &=\nu(A)+k        &       \\
&\Leftrightarrow          &  n-\nu(D)&=n-\nu(A)-k  & \\
&\Leftrightarrow   &  \text{rank}(D)+k &= \text{rank}(A)-k &     \\
&\Leftrightarrow &  \text{rank}(A) &= \text{rank}(D)+2k &
\end{align*} 
Thus (a) and (b) are equivalent.

Now note \begin{equation*}
\text{rank}(A) \leq k+\text{rank}\left(\begin{bmatrix} C \\ D \end{bmatrix}\right) \leq k+k+\text{rank}(D) = 2k + \text{rank}(D) 
\end{equation*}
with equality if and only if (b) holds.  Hence (b) and (c) are equivalent.
\end{proof}

We use Theorem \ref{Theorem1} to give sufficent conditions for a P-set of a principal submatrix to be a P-set of the entire matrx.

\begin{corollary}\label{subsetP-set}
Let $A$ be an $n \times n$ symmetric matrix, $\alpha$ be a set of P-vertices of $A$, and $\beta$ be a subset of $\{1,2,\dots, n\}$ such that $\alpha \subseteq \beta$ and the rows of $A$ indexed by $\beta$ are a basis of RS$(A)$.  If $\alpha$ is a P-set of $A[\beta]$, then $\alpha$ is a P-set of $A$.
\end{corollary}

\begin{proof}  Assume $\alpha$ is a P-set of $A[\beta]$.  Without loss of generality, $$A = \left[ \begin{array} {c | c | c}
B & C & E \\ \hline
C^T & D & F \\ \hline
E^T & F^T & G
\end{array} \right], \; A[\alpha] = B \text{ and } A[\beta] = \left[ \begin{array}{c | c} B & C \\ \hline C^T & D \end{array} \right].$$  Since $\alpha$ is a P-set of $A[\beta]$, by Theorem \ref{Theorem1}, $\{x: x^TC \in \text{RS}(D)\} = \{0\}$.  Hence $\{x: x^T\begin{bmatrix} C & E \end{bmatrix} \in \text{RS}(\begin{bmatrix} D &  F \end{bmatrix})\} = \{0\}$.  

We now show that $$\text{RS}(\begin{bmatrix} D & F \end{bmatrix}) = \text{RS}\left( \begin{bmatrix} D & F \\ F^T & G \end{bmatrix} \right).$$  
Since the rows of $A$ indexed by $\beta$ are a basis of RS$(A)$, $$\text{RS}(\begin{bmatrix} E^T & F^T & G  \end{bmatrix}) \subseteq \text{RS}\left(\begin{bmatrix} B & C & E \\ C^T & D & F \end{bmatrix}\right).$$  
Since $\alpha$ is a set of P-vertices of $A$, by Proposition \ref{independent}, no row of $\begin{bmatrix} B & C & E \end{bmatrix}$ is a linear combinatation of other rows of $A$.  It follows that $\text{RS}(\begin{bmatrix} E^T & F^T & G  \end{bmatrix}) \subseteq \text{RS}\left( \begin{bmatrix}C^T & D & F \end{bmatrix} \right)$.  In particular, each row of $\begin{bmatrix} F^T & G \end{bmatrix}$ is in $\text{RS}(\begin{bmatrix} D & F \end{bmatrix})$ and hence $$\text{RS}(\begin{bmatrix} D & F \end{bmatrix}) = \text{RS}\left(\begin{bmatrix} D & F \\ F^T & G \end{bmatrix}\right).$$
Therefore $$ \left\{ x: x^T\begin{bmatrix} C & E \end{bmatrix} \in \text{RS}\left( \begin{bmatrix} D & F \\ F^T & G \end{bmatrix} \right)\right\} = \{0\}$$ and thus, by Theorem \ref{Theorem1}, $\alpha$ is a P-set of $A$.
\end{proof}

We now show that all pairs suffice in the case the matrix is nonsingular.

\begin{theorem}\label{nonsing2toall}
Let $M$ be a nonsingular, symmetric $n \times n$ matrix and $\alpha$ be a set of cardinality at least two such that each pair of $\alpha$ is a P-set of $M$.  Then $\alpha$ is a P-set of $M$.
\end{theorem}

\begin{proof}
Let $N = M^{-1}$.  Since each non-empty subset of a P-set is a P-set, the hypotheses imply that $\det M(i) = 0$ for each $i \in \alpha$.  By Theorem \ref{Jacobi}, $\det N[i] = 0$ for each $i \in \alpha$.  Thus $n_{ii} = 0$ for each $i \in \alpha$.  Also, for each subset $\gamma$ of $\alpha$ of cardinality $2$, we know that $\det M(\gamma) = 0$.  Thus $\det N[\gamma] = 0$.  Since the diagonal entries of the $2 \times 2$, symmetric matrix $N[\gamma]$ are $0$, $N[\gamma] = 0$ for each such $\gamma$.  Hence $N[\alpha] = 0$.  

Now consider a subset $\tau$ of $\{1,2,\dots,n\}\setminus \alpha$ of cardinality at most $|\alpha|-1$.  It is easy to verify that the columns of $N[\tau, \alpha]$ are linearly dependent, and hence the columns of $N[\tau \cup \alpha]$ indexed by $\tau$ are linearly dependent.  We conclude that $\det N[\tau \cup \alpha] = 0$.  Thus, by Theorem \ref{Jacobi}, every principal minor of $M(\alpha)$ of order at least $n-2|\alpha|+1$ is $0$.  Since $M(\alpha)$ is a symmetric matrix of order $n-|\alpha|$, the nullity of $M(\alpha)$ is greater than $|\alpha|-1$.  We conclude that $\alpha$ is a P-set of $M$.  
\end{proof}

Lastly, we prove the main result of the note.

\begin{theorem}\label{2toall}
Let $A$ be an $n \times n$ symmetric matrix and $\alpha$ a set of indices of cardinality at least two.  Then $\alpha$ is a P-set of $A$ if and only if every pair of $\alpha$ is a P-set of $A$.
\end{theorem}

\begin{proof}
First assume $\alpha$ is a P-set of $A$.  Since every subset of a P-set is a P-set, each pair is a P-set of $A$.  

Conversely, assume each pair of $\alpha$ is a P-set of $A$.  By Proposition \ref{independent}, the rows of $A$ indexed by $\alpha$ are linearly independent.  Thus there exists a set $\beta$ such that $\alpha \subseteq \beta$, $|\beta| = n-\nu(A)$, and $A[\beta]$ is nonsingular.  It follows that the rows of $A$ indexed by $\beta$ are a basis of the row space of $A$.

Consider a subset $\gamma$ of $\alpha$ of cardinality $2$.  By assumption, $A(\gamma)$ has nullity $\nu(A)+2$.  Thus $A[\beta\setminus\gamma]$, which is obtained from $A(\gamma)$ by deleting $\nu(A)$ rows and columns, has nullity at least $\nu(A)+2-\nu(A)= 2$.  Hence for each pair $\gamma$ of $\alpha$, $\gamma$ is a P-set of $A[\beta]$.  By Theorem \ref{nonsing2toall}, $\alpha$ is a P-set of the nonsingular $A[\beta]$.  By Corollary \ref{subsetP-set}, $\alpha$ is a P-set of $A$.
\end{proof}

Thus if $A$ is a symmetric matrix and $\alpha$ is a set of indices of cardinality at least two, then the condition that $\nu(A)+2 = \nu(A\{i,j\})$ for each subset $\{i,j\}$ of $\alpha$ implies $\alpha$ is a P-set of $A$.  We note that if $\alpha$ is a set of P-vertices of $A$, then in order to show $\alpha$ is a P-set of $A$, it suffices to show the weaker condition that $\nu(A) < \nu(A(\{i,j\})$ for every subset $\{i,j\}$ of $\alpha$. This is seen as follows.  In \cite[Proposition 4.7]{JS}, it is shown that if $i$ and $j$ are P-vertices of $A$, then $\nu(A)-\nu(A(\{i,j\}) \in \{-2,0\}$.  Combining this with Theorem \ref{2toall} shows that if $\alpha$ is a set of P-vertices of $A$, then $\alpha$ is a P-set of $A$ if and only if $\nu(A) < \nu(A(\{i,j\}))$ for every subset $\{i,j\}$ of $\alpha$. 







\begin{thebibliography}{00}


 \bibitem{HJ}
 R. Horn and C.R. Johnson.
 \newblock {\em Matrix Analysis}. 2nd ed.
 \newblock Cambridge University Press, New York, 2013.


 \bibitem{JDSSW}
 C.R. Johnson, A. Leal Duarte, C.M. Saiago, B.D. Sutton, and A.J. Witt.
 \newblock On the relative position of multiple eigenvalues in the spectrum of an Hermitian matrix with a given graph.
 \newblock {\em Linear Algebra and its Applications}, 363 (2003) 147-159. 

 \bibitem{JS}
 C.R. Johnson and B.D. Sutton.
  \newblock Hermitian matrices, eigenvalue multiplicities, and eigenvector components.
 \newblock {\em SIAM Journal on Matrix Analysis and Applications}, 26 (2004) 390-399.

 \bibitem{KS}
 I.J. Kim and B.L. Shader.
 \newblock On Fiedler- and Parter-vertices of acyclic matrices.
 \newblock {\em Linear Algebra and its Applications}, 428 (2008) 2601-2613.

 \bibitem{KS1}
 I.J. Kim and B.L. Shader.
 \newblock Non-singular acyclic matrices.
 \newblock {\em Linear and Multilinear Algebra}, 57:4 (2009) 399-407.

\bibitem{NS2}
C.G. Nelson and B.L. Shader.
\newblock Maximal P-sets of matrices whose graph is a tree.
\newblock {\em Linear Algebra and its Applications}, under review.




\end{thebibliography}



\end{document}